\documentclass[11pt,a4paper]{article}
\usepackage[left=3.5cm, right=3.5cm, top=2cm, bottom=2.5cm]{geometry}
\usepackage[latin1]{inputenc}
\usepackage[T1]{fontenc}
\usepackage{graphicx}
\usepackage{amsmath}
\usepackage{amssymb}
\usepackage{amsthm}
\newtheorem{theorem}{Theorem}
\newtheorem{lemma}{Lemma}
\newtheorem{corollary}{Corollary}
\usepackage{bm}
\usepackage{natbib}
\usepackage[english]{babel}

\begin{document}
\begin{center}
  \Large
  \textbf{Optimal designs for $K$-factor two-level models with first-order interactions on a
    symmetrically restricted design region}
\end{center}
\begin{center}
  \textbf{Fritjof Freise\footnote{corresponding author}\\
    TU Dortmund University, Department of Statistics,\\
    Vogelpothsweg 87, 44227 Dortmund, Germany,\\
    e-mail: fritjof.freise@tu-dortmund.de}
\end{center}
\begin{center}
  \textbf{Rainer Schwabe\\
    University of Magdeburg, Institute for Mathematical Stochastics,\\
    Universit\"atsplatz 2, 39106 Magdeburg, Germany,\\
    e-mail: rainer.schwabe@ovgu.de}
\end{center}

\begin{abstract}
  We develop $D$-optimal designs for linear models with first-order
  interactions on a subset of the $2^K$ full factorial design region,
  when both the number of factors set to the higher level and the
  number of factors set to the lower level are simultaneously bounded
  by the same threshold.  It turns out that in the case of narrow
  margins the optimal design is concentrated only on those design
  points, for which either the threshold is attained or the numbers of
  high and low levels are as equal as possible. In the case of wider
  margins the settings are more spread and the resulting optimal
  designs are as efficient as a full factorial design.  These findings
  also apply to other optimality criteria.
\end{abstract}

\section{Introduction}
In the situation of item calibration for psychological tests the goal
is to estimate the difficulty of new items or of the effect of certain
attributes of an item.  In the rule based approach for item generation
the difficulty of an item can be split into several components
corresponding to different rules, which may or may not be applied.
Conceptually this results in a linear predictor for the difficulty
based on the active rules.  Typically an item will become more
difficult, if more rules are used in its construction.  To avoid so
called ceiling or bottom effects, i.e.\ an item is answered correctly
or wrong, respectively, by all the participants, items should not be
too difficult nor too easy.  Moreover, the assumed model may become
inappropriate if too extreme items, i.e.\ items with too many or items
with too few rules, are used.  All this may cause the necessity to
restrict the number of rules used from below and above.  In this
scenario standard designs like full or standard fractional factorials
are no longer applicable.

The application, which motivated this research, was in the latter
regime. Here the total number of rules was six and items with two up
to four active rules were to be used.  This lead directly to the
question in which cases fully efficient designs exist, i.e.\ designs
which are as efficient as the full factorial design. These results,
which will be presented in Theorem~\ref{thm:DoptEffAsFactorial}, can
also be the basis for the search for and construction of irregular
fractions.

For the case of a linear predictor which contains only main effects,
optimal designs have been characterized in
\citet*{FreiseHollingSchwabe:2018prea}.  We will follow the lines
indicated there to extend the results to situations where first-order
(two-factor) interactions have to be incorporated.  The approach used
for characterizing optimal designs is based on invariance and
equivariance considerations with respect to natural symmetries in both
the model and the design region.  This approach has successfully been
applied also in other settings like spring balance weighing
\citep*[see][]{FilovaHarmanKlein:2011a}, and the findings obtained
here may be transferred to other situations, where there are natural
constraints on the simultaneous occurrence of high or of low levels,
respectively.  To keep notations simple and to reduce technicalities,
we will confine ourselves to the case of symmetric constraints.

The manuscript is organized as follows.  After a brief description of
the model and basic concepts of design and invariance in
Section~\ref{sec:ModelAndInfo}, the general structure of the
information matrix is presented in Section~\ref{sec:StructureOfInfo}.
In Section~\ref{sec:symmetric-case} we specialize to the symmetric
case and present characterizations of optimal designs in
Section~\ref{sec:DoptDesign}.  The manuscript concludes with a brief
discussion on potential extensions in Section~\ref{sec:discussion} and
is augmented by tables of optimal designs for up to $22$ rules.  All
proofs are deferred to an appendix.

\section{Model Description, Information and Invariance}
\label{sec:ModelAndInfo}
We consider an experiment in which $N$ items are presented and $K$
rules may be used to construct an item. Then an item can be
characterized by its settings (design points)
$\bm{x}=(x_{1},\ldots,x_{K})^\top\in\{-1,+1\}^K$, where $x_{k}= +1$,
if the $k$-th rule is used in the construction of the item, and
$x_{k}=-1$, if the $k$-th rule is absent.

Even though using $0$ and $1$ for inactive and active rules would seem
to be more natural in our application, the present parametrization was
chosen out of convenience. Especially the information matrices have a
more intuitive representation. Note also, that for the $D$-criterion,
which is primarily of interest here, reparametrization has no
influence on the optimal design.

We assume an underlying analysis of variance model with first-order
(two-factor) interactions for the observations, i.e.\ for the item
scores, $Y_1,\ldots,Y_N$ obtained for $N$ items described by their
settings $\bm{x}_1,\ldots,\bm{x}_N$.  This model can be written as
\begin{equation*}
  Y_{i}=\bm{f}(\bm{x}_{i})^\top\bm{\beta}+\varepsilon_{i}\,,
\end{equation*}
$i=1,\ldots,N$, where the regression function $\bm{f}$ is specified by
\begin{equation*}
  \bm{f}(\bm{x})
  =\begin{pmatrix}
    1&\bm{x}^\top&\bm{\tilde{x}}^\top
  \end{pmatrix}^\top
\end{equation*}
and $\bm{\tilde{x}}=(x_{1}x_{2},\ldots,x_{K-1}x_{K})^\top$ collects
the interactions.  The difficulties of the items are, hence, specified
by the $p$-dimensional parameter vector $\bm{\beta}$, which includes a
constant term, $K$ parameters corresponding to the main effects and
${K \choose 2}= K(K-1)/2$ parameters for the first-order (two-factor)
interactions.  Thus the total number of parameters amounts to
$p=1+K(K+1)/2$.  The error terms
$\varepsilon_{1},\ldots,\varepsilon_{N}$ are assumed to be centered,
uncorrelated and homoscedastic, with
$\mathrm{Var}(\varepsilon_{i})=\sigma^{2}$.

Note, that in the calibration example several items may be solved by
the same person and, hence, the assumption of uncorrelated
observations may be violated. Nonetheless, uncorrelatedness of the
observations will be assumed here, to keep the results simple.

The quality of an experiment may be measured in terms of the
information matrix
\begin{equation*}
  \mathbf{M}(\bm{x}_1,\ldots,\bm{x}_N)
  = \sum_{i=1}^{N}\bm{f}(\bm{x}_{i})\bm{f}(\bm{x}_{i})^\top\,.
\end{equation*}
It is well-known that in the case of full rank the variance-covariance
matrix of the least squares estimator is proportional to the inverse
of the information matrix.  Thus the aim of an optimal design is to
find optimal settings $\bm{x}_1,\ldots,\bm{x}_N$, which minimize the
variance-covariance matrix or, equivalently, maximize the
information matrix in a suitable sense.  As a uniform optimization of
the matrices in the Loewner ordering of nonnegative definiteness is
impossible, one has to choose some real-valued information functional
on the matrices.  Here we will adopt the most popular criterion of
$D$-optimality, which aims at maximizing the determinant of the
information matrix. This can be interpreted as minimization of the
volume of the confidence ellipsoid for the whole parameter vector
under the additional assumption of Gaussian error terms.  However, the
results obtained may be generalized also to other criteria which share
the invariance properties described below.  Examples include $A$- and
$E$-optimality \citep[see also][]{FilovaHarmanKlein:2011a}.

To facilitate the search for optimal designs we will make use of the
concepts of approximate designs and the corresponding well-developed
theory \citep[see for example][]{Silvey:1980}: An approximate design
\begin{equation*}
  \xi=\left\{
    \begin{matrix}
      \bm{x}_{1}&\ldots&\bm{x}_{n}\\
      w_{1}&\ldots&w_{n}
    \end{matrix}
  \right\}
\end{equation*}
is defined by $n$ mutually different settings
$\bm{x}_1,\ldots,\bm{x}_n$ with associated weights $w_{i} \geq 0$
satisfying $\sum_{i=1}^{n}w_{i} = 1$.  The settings $\bm{x}_i$ in the
approximate design $\xi$ correspond to the mutually different settings
occurring in the exact design $(\bm{x}_1,\ldots,\bm{x}_N)$ of sample
size $N$ specified before, and the weights represent the corresponding
frequencies $w_i=N_i/N$ of occurrence, where $N_i$ is the number of
replications of $\bm{x}_i$.  However, in general, for an approximate
design the condition on the weights being multiples of $1/N$ is
relaxed.

The weighted (per observation) information matrix of an approximate
design $\xi$ is then defined by
\begin{equation*}
  \mathbf{M}(\xi)
  = \sum_{i=1}^{n}w_{i}\bm{f}(\bm{x}_{i})\bm{f}(\bm{x}_{i})^\top\;.
\end{equation*}
For an exact design of sample size $N$ this definition coincides with
the standardized information matrix
$1/N\cdot\mathbf{M}(\bm{x}_1,\ldots,\bm{x}_N)$.

The settings $\bm{x}_i$ may be chosen from a design region
$\mathcal{X}$ of possible settings, for which the model equation
holds.  In the present situation we assume that the design region
$\mathcal{X}\subset \{-1,+1\}^K$ is restricted by the possible
combinations of rules, i.e.\ by the number of rules simultaneously
applied.  Denote by $d(\bm{x}) = (\sum_{k=1}^{K}x_{k} + K) / 2$ the
number of active rules, i.e.\ the number of entries equal to $+1$ in a
design point $\bm{x}$ and let $L$ and $U$ be the minimal and maximal
number, respectively, of active rules allowed.  Then the design region
is defined as
\begin{align*}
  \mathcal{X}_{L,U}
  &= \{\bm{x}\in\{-1,+1\}^K ;\ L \leq d(\bm{x}) \leq U\}.
\end{align*}
The condition on $d(\bm{x})$ can be rewritten as
$2L-K \leq \sum_{k=1}^{K}x_{k} \leq 2U-K$.

One of the major results in the context of approximate design theory is
the equivalence theorem of \citet{KieferWolfowitz:1960}. It states
that a design $\xi^*$ is $D$-optimal if and only if
\begin{equation}
  \label{eq:equivalencetheorem}
  \bm{f}(\bm{x})^\top\mathbf{M}(\xi^*)^{-1}\bm{f}(\bm{x}) \leq p
\end{equation}
for all design points $\bm{x}$ in the design region $\mathcal{X}$.
The left-hand side is the so called sensitivity function and will be
denoted by $\psi(\bm{x})$.

We will make use of invariance properties to reduce the complexity of
the optimization problem. See for example \cite{Pukelsheim:1993} and
\cite{Schwabe:1996} for details and further references.  The design
problem on the vertices of the hypercube $\{-1,+1\}^{K}$ is invariant
under permutation of entries in the design point, i.e.\ the
permutation of rules. These permutations constitute a transformation
group on the full $2^K$ hypercube as well as on the restricted design
region $\mathcal{X}_{L,U}$. On the full hypercube there are $K+1$
orbits under the permutation, denoted
$\mathcal{O}_{0},\ldots,\mathcal{O}_{K}$. Each orbit $\mathcal{O}_{k}$
consists of all items with $k$ active rules or, equivalently, of the
design points with $d(\bm{x})=k$ entries equal to $+1$,
$k=0,1,\ldots,K$. Moreover, the regression function is linearly
equivariant with respect to these permutations, i.e.\ for each
permutation $g$ there is a matrix $\mathbf{Q}_g$ such that
$\bm{f}(g(\bm{x}))=\mathbf{Q}_g\bm{f}(\bm{x})$ for all $\bm{x}$. Since
the $D$-criterion is invariant under these conditions, there exists a
$D$-optimal design uniform on the orbits, i.e.\ all settings within
one orbit get the same weight. Such designs will also be called
invariant with respect to permutations, since they remain unchanged,
if the support is transformed.

Denote the uniform design on the orbit $\mathcal{O}_k$ with $k$ active
rules by
$\bar{\xi}_{k}$. The corresponding information matrix $\mathbf{M}(\bar{\xi}_k)$
measuring the information of the orbit $\mathcal{O}_k$ is
given by
\begin{equation*}
  \mathbf{M}(\bar{\xi}_{k}) = C(K,k)^{-1}
  \sum_{\bm{x}\in \mathcal{O}_{k}}\bm{f}(\bm{x})\bm{f}(\bm{x})^\top\,,
\end{equation*}
where here and later on we use the notation
$C(K,k)$ for the binomial coefficient ${K \choose k}$,
when this is more convenient.

Any invariant design $\bar{\xi}$ is a convex combination of these
single orbit designs:
\begin{equation*}
\bar{\xi}
  =\sum_{k=0}^{K}\bar{w}_{k}\bar{\xi}_{k}
\end{equation*}
with weights $\bar{w}_{k} \geq 0$, $\sum_{k=0}^{K}\bar{w}_{k} = 1$.
Thus the information
matrix of an invariant design $\bar{\xi}$ becomes
\begin{equation*}
  \mathbf{M}(\bar{\xi})
  =\sum_{k=0}^{K}\bar{w}_{k}\mathbf{M}(\bar{\xi}_{k})\,.
\end{equation*}
Hence, the optimization can be confined to finding optimal weights 
$\bar{w}_{k}$ on the orbits.
Note that the restricted design region $\mathcal{X}_{L,U}$ just 
imposes the side conditions $\bar{w}_k=0$ for $k<L$ or $k>U$.

In the following we further consider the special case of design
regions of the form $\mathcal{X}_{L,K-L}$, i.e.\ symmetric thresholds
$U=K-L$, which bound both the number $d(\bm{x})$ of active rules and
the number $K-d(\bm{x})$ of inactive rules by the same threshold $U$
form above.  Under this condition we may utilize a further symmetry
with respect to simultaneous sign change of all entries in the vector
$\bm{x}$ representing the design point, i.e.\ $d(\bm{x})$ active rules
are changed to become inactive and vice versa. The corresponding
transformation matches orbits $\mathcal{O}_k$ and $\mathcal{O}_{K-k}$
for $0\leq k \leq K/2$.  Then, under the full group of permutations
and symmetry, there are $K/2+1$ or $(K+1)/2$ orbits, if $K$ is even or
odd, respectively, which consist of all items with $k$ or $K-k$ active
rules. It follows, that for the invariant optimal design
$\bar{w}_{k}=\bar{w}_{K-k}$ can be chosen.

\section{General Structure of the Information Matrix}
\label{sec:StructureOfInfo}
As usual under the current parametrization the entries of the 
information matrix $\mathbf{M}(\xi)$ can be considered as moments 
\begin{equation*}
  \sum_{i=1}^{n}w_{i}x_{ij}^{r}x_{ik}^{s}x_{i\ell}^{t}x_{im}^{u}\,,
  \quad j, k, \ell, m \in\{1,\ldots,K\},\,
  r,s,t,u\in\{0,1,2\}\,,
\end{equation*}
with respect to the design $\xi$, where the $x_{ik}$ denote the
entries in $\bm{x}_i$.  The diagonal entries of the information matrix
are given by
\begin{equation*}
  \sum_{i=1}^{n}w_{i}1 = \sum_{i=1}^{n}w_{i}x_{ik}^{2}
  =\sum_{i=1}^{n}w_{i}x_{ij}^{2}x_{ik}^{2} = 1\,,
  \quad j, k \in\{1,\ldots,K\},\, j\neq k\,.
\end{equation*}
Moreover, for designs invariant with respect to permutations
there are only four further potentially different entries 
for the off-diagonal elements.
In the first row and first column  the first moments for the main
effects are
\begin{equation*}
  m_{1}(\bar{\xi})= \sum_{i=1}^{n}w_{i}x_{ik}, \quad k=1,\ldots,K\,,
\end{equation*}
while the first moments for the interactions as well as the mixed 
moments for the main effects become
\begin{equation*}
  m_{2}(\bar{\xi})=\sum_{i=1}^{n}w_{i}x_{ij}x_{ik}\,,
  \quad 1\leq j < k \leq K \,.
\end{equation*}
For the mixed main effect/interaction moments we obtain
$m_1(\bar{\xi})$, when the main effect factor is involved in the 
interaction, and
\begin{equation*}
  m_{3}(\bar{\xi})=\sum_{i=1}^{n}w_{i}x_{ij}x_{ik}x_{i\ell}\,,
  \quad 1\leq j < k < \ell \leq K\,,
\end{equation*}
otherwise. For the mixed interaction moments we get 
$m_2(\bar{\xi})$, when there is a common factor in both 
interactions, and
\begin{equation*}
  m_{4}(\bar{\xi})=\sum_{i=1}^{n}w_{i}x_{ij}x_{ik}x_{i\ell}x_{im}\,,
  \quad 1\leq j < k < \ell < m \leq K
  \,,
\end{equation*}
if all factors are different.

Even though $m_{1}$ through $m_{4}$ depend on the design $\bar{\xi}$,
we will omit the argument for the sake of brevity, when this does not
cause confusion.  With this notation the information matrix becomes
\begin{equation*}
  \mathbf{M}(\bar{\xi}) = \begin{pmatrix}
    1 & m_{1} \bm{1}_{K}^\top & m_{2}
    \bm{1}_{C(K,2)}^\top\\
    m_{1} \bm{1}_{K} & \mathbf{M}_{11} & \mathbf{M}_{12}\\
    m_{2} \bm{1}_{C(K,2)} & \mathbf{M}_{12}^\top & \mathbf{M}_{22}
  \end{pmatrix}\,,
\end{equation*}
where $\bm{1}_\ell$ denotes a $\ell$-dimensional vector with all entries equal to $1$.
The block $\mathbf{M}_{11}$ corresponding to the main effects has the form
\begin{equation*}
  \mathbf{M}_{11}
  = (1 - m_{2}) \mathbf{I}_{K} + m_{2}
  \bm{J}_{K}\,,
\end{equation*}
where $\mathbf{I}_\ell$ and $\mathbf{J}_\ell$ denote 
the $\ell\times\ell$ identity matrix and the $\ell\times\ell$
matrix with all entries equal to $1$, respectively.
The blocks involving interactions are given by
\begin{equation*}
  \mathbf{M}_{12}
  = (m_{1}-m_{3})\mathbf{S}_{K}^\top + m_{3}
  \bm{1}_{K}\bm{1}_{C(K,2)}^\top
\end{equation*}
and
\begin{equation*}
  \mathbf{M}_{22}
  = (1 - 2m_{2} + m_{4})\mathbf{I}_{C(K,2)}
  +(m_{2} - m_{4})
  \mathbf{S}_{K}\mathbf{S}_{K}^\top
  +m_{4}\bm{J}_{C(K,2)}\,,
\end{equation*}
where the $C(K,2)\times K$-matrix $\mathbf{S}_{K}$ contains only
entries equal to $1$ or $0$ indicating whether the corresponding main
effect is involved in the associated interaction -- or not. For
illustrative purposes we exhibit this matrix in the case $K=6$:
\setcounter{MaxMatrixCols}{15}
\begin{equation*}
  \mathbf{S}_{6}^\top
  = \begin{pmatrix}
    1&1&1&1&1&0&0&0&0&0&0&0&0&0&0\\
    1&0&0&0&0&1&1&1&1&0&0&0&0&0&0\\
    0&1&0&0&0&1&0&0&0&1&1&1&0&0&0\\
    0&0&1&0&0&0&1&0&0&1&0&0&1&1&0\\
    0&0&0&1&0&0&0&1&0&0&1&0&1&0&1\\
    0&0&0&0&1&0&0&0&1&0&0&1&0&1&1
  \end{pmatrix}\,.
\end{equation*}

Because the matrix $\mathbf{S}_K$ has $2$ entries equal to one in each
row and $K-1$ in each column, we get
\begin{equation*}
  \mathbf{S}_{K}\bm{1}_{K} = 2\cdot\bm{1}_{C(K,2)}
  \quad\text{and}\quad
  \mathbf{S}_{K}^\top\bm{1}_{C(K,2)} = (K-1)\bm{1}_{K}
\end{equation*}
and, furthermore,
\begin{equation*}
  \mathbf{S}_{K}^\top\mathbf{S}_{K}
  = (K-2)\mathbf{I}_{K} + \bm{J}_{K}\,.
\end{equation*}
For further use we note that the vector $\bm{1}_{K}$ is an eigenvector
of $\mathbf{S}_{K}^\top\mathbf{S}_{K}$ with corresponding eigenvalue
$2(K-1)$, and the remaining $K-1$ eigenvalues of
$\mathbf{S}_{K}^\top\mathbf{S}_{K}$ are all equal to $K-2$ with
corresponding eigenvectors, which are orthogonal to $\bm{1}_{K}$.  As
a consequence, these values are also the non-zero eigenvalues of
$\mathbf{S}_{K}\mathbf{S}_{K}^\top$, where the eigenvalue $2(K-1)$
corresponds to the eigenvector $\bm{1}_{C(K,2)}$.  By this observation,
the eigenvalues of the diagonal block $\mathbf{M}_{22}$ of the
information matrix associated with the interactions can be determined
as
$1-2m_{2} +m_{4}+ 2(K-1)(m_{2}-m_{4}) + C(K,2)m_{4}
=1+2(K-2)m_{2}+(K-2)(K-3)m_{4}/2 $ with multiplicity one, corresponding
to the eigenvector $\bm{1}_{C(K,2)}$,
$ 1-2m_{2} +m_{4}+ (K-2)(m_{2}-m_{4}) = 1+(K-4)m_{2} - (K-3) m_{4} $
with multiplicity $K-1$, corresponding to the remaining eigenvectors of
$\bm{S}_{K}\bm{S}_{K}^\top$, and $ 1-2m_{2}+m_{4} $ with multiplicity
$C(K,2)-K=K(K-3)/2$, corresponding to the eigenvectors orthogonal to
the previous ones.
  
As has be seen before, the information matrix of an invariant design
can be written as a weighted sum of the information matrices of the
orbits. For the information matrices of the orbits the entries can be
calculated combinatorially by counting the number of terms in the sums
which are equal to $+1$ and $-1$, respectively. While the diagonal
entries in the information matrices are all equal to $1$, the
off-diagonal entries are determined by the moments, which, in general,
can be derived as
\begin{equation}
  \label{eq:4}
  m_{j}(\bar{\xi}_{k})
  =C(K,k)^{-1}\sum_{i=0}^{j}(-1)^{i+j}C(j,i) C(K-j,k-i)\,.
\end{equation}
In particular, we obtain
\begin{align}
  \label{eq:def:mbar1k}
  m_{1}(\bar{\xi}_{k})
  &
    = \frac{2 k - K}{K} \,, \\
  \label{eq:def:mbar2k}
  m_{2}(\bar{\xi}_{k})
  &
    = \frac{(2 k - K)^2  - K}{K(K - 1)} \,, \\
  \label{eq:def:mbar3k}
  m_{3}(\bar{\xi}_{k})
  &
  =\frac{(2 k - K)^3 - (3 K-2)(2k-K)}{K(K - 1)(K - 2)} \,, 
  \quad \text{and}\\
  \label{eq:def:mbar4k}
  m_{4}(\bar{\xi}_{k})
  &=\frac{(2k-K)^4 - (6K-8)(2k-K)^2 + 3K(K-2)}{K(K-1)(K-2)(K-3)}
\end{align}
for $j \leq K$, and $m_{j}(\bar{\xi}_{k})=0$ otherwise.  Obviously,
for $j \leq K$, the moments $m_{j}(\bar{\xi}_{k})$ are polynomials in
$k$ of degree $j$ with positive leading terms, which are symmetric
with respect to $K/2$.
\section{Symmetric case}
\label{sec:symmetric-case}
First note that 
$m_{j}(\bar{\xi}_{K-k})=-m_{j}(\bar{\xi}_{k})$
for $j$ odd and
$m_{j}(\bar{\xi}_{K-k})=m_{j}(\bar{\xi}_{k})$
for $j$ even, respectively.

Additionally in the case of symmetric constraints ($L+U=K$) we have
equal weights $\bar{w}_{K-k}=\bar{w}_{k}$ for symmetric invariant
designs.  For these designs follows immediately that their odd moments
vanish ($m_{1}=m_{3}=0$) and hence we have
$\mathbf{M}_{12}=\mathbf{0}$ for the off-diagonal block in the
information matrix.  Thus the information matrix simplifies to a
chess-board structure,
\begin{equation}
  \label{eq:infomatrixsymcase}
  \mathbf{M}(\bar{\xi})
  =\begin{pmatrix}
    1&\bm{0}&m_{2}\bm{1}_{C(K,2)}^\top\\
    \bm{0}&\mathbf{M}_{11}&\bm{0}\\
    m_{2}\bm{1}_{C(K,2)}&\bm{0}&\mathbf{M}_{22}
  \end{pmatrix} \,,
\end{equation}
where $\bm{0}$ denotes a vector or a matrix of appropriate size with
all entries equal to $0$.

The determinant of this matrix can be calculated by using standard
formulae as
\begin{align*}
  \det\left(\mathbf{M}(\bar{\xi})\right)
      &=\det(\mathbf{M}_{11})
        \det(\mathbf{M}_{22}-m_{2}^{2}\bm{J}_{C(K,2)})\,.
\end{align*}
For the first determinant on the right-hand side we have
\begin{align*}
\det(\mathbf{M}_{11})
      &=(1+(K-1)m_{2})(1-m_{2})^{K-1} \,.
\end{align*}
For the second determinant observe that the eigenvalue of the
matrix $\mathbf{M}_{22}-m_{2}^{2}\bm{J}_{C(K,2)}$ associated with the
eigenvector $\bm{1}_{C(K,2)}$ equals that of $\mathbf{M}_{22}$ reduced
by $m_{2}^{2}C(K,2)$, while the remaining eigenvalues stay the same.
As a consequence the determinant can be obtained as
\begin{align*}
  \det(&\mathbf{M}_{22}-m_{2}^{2}\bm{J}_{C(K,2)})\\
         &=(1+2(K-2)m_{2}
         +\textstyle{\frac{1}{2}}(K-2)(K-3)m_{4}
         - \textstyle{\frac{1}{2}}K(K-1)m_{2}^2)\\
       &\qquad\times
         \left(1+(K-4)m_{2} - (K-3) m_{4}\right)^{K-1}\\
       &\qquad\qquad\times
         (1-2m_{2}+m_{4})^{K(K-3)/2}\,.
\end{align*}

The occurring eigenvalues are all nonnegative because the information
matrix is nonnegative definite.  For estimability of all parameters
it has to be shown that these eigenvalues are all positive.

The following lemma establishes necessary and sufficient conditions
for the information matrix of an invariant symmetric design to be
nonsingular.  For this we denote by
$\tilde\xi_k=(\bar\xi_k+\bar\xi_{K-k})/2$ the uniform design on the
symmetric orbit
$\tilde{\mathcal{O}}_{k}=\mathcal{O}_{k}\cup\mathcal{O}_{K-k}$.  Note
that for $K$ even the symmetric orbit for $k=K/2$ is degenerate,
$\tilde{\mathcal{O}}_{K/2}=\mathcal{O}_{K/2}$.

\begin{lemma}
  \label{lem:regularity}
  Let $K\geq 2$.  For a symmetric invariant design $\bar\xi$ the
  information matrix $\mathbf{M}(\bar{\xi})$ is nonsingular if only if
  $\bar{\xi}$ is supported on, at least, two distinct symmetric
  orbits $\tilde{\mathcal{O}}_k$ and $\tilde{\mathcal{O}}_{\ell}$,
  $0\leq k < \ell \leq K/2$, where either $K\leq 3$, $k>0$ or
  $2\leq\ell<K/2$.
\end{lemma}
Note that it follows from Lemma~\ref{lem:regularity} that for $K=2$ and $K=3$ 
the full $2^K$ factorial would be required 
for estimability of the parameters,
and that for a symmetric invariant design with   more than two different 
symmetric orbits estimability of all parameters is always ensured.

If the information matrix $\mathbf{M}(\bar{\xi})$ is nonsingular, 
its inverse has the same chess-board structure,
\begin{equation}
  \label{eq:inverseofinfo}
  \mathbf{M}(\bar{\xi})^{-1}
  =\begin{pmatrix}
    c_{0}&\bm{0}&-c_{2}\bm{1}_{C(K,2)}^\top\\
    \bm{0}&\mathbf{M}_{11}^{-1}&\bm{0}\\
    -c_{2}\bm{1}_{C(K,2)}&\bm{0}&\mathbf{C}_{22}\\
  \end{pmatrix}\,,
\end{equation}
where
\begin{align*}
  c_{2} &= \frac{2m_{2}}{
          2 +4(K-2) m_{2} +
          (K-2)(K-3)m_{4}-K(K-1)m_{2}^2}\,,\\
  c_{0} &= 1 + c_{2}C(K,2)m_{2}\,,\\
  \mathbf{M}_{11}^{-1}
        &= \frac{1}{1-m_{2}} \left(\mathbf{I}_{K}
          -\frac{m_{2}}{1+(K-1)m_{2}}\mathbf{J}_{K}\right)\,,\\
\end{align*}
and
\begin{equation*}
  \mathbf{C}_{22}
        =\frac{1}{1-2m_{2}+m_{4}}\left(
          \mathbf{I}_{C(K,2)}
          -\delta_{\mathbf{S}}\mathbf{S}_{K}\mathbf{S}_{K}^\top
          -\delta_{\mathbf{J}}\mathbf{J}_{C(K,2)}\right)\,.
\end{equation*}
The coefficients $\delta_{\mathbf{S}}$ and $\delta_{\mathbf{J}}$ are given by
\begin{equation*}
  \delta_{\mathbf{S}}=\frac{m_{2}-m_{4}}{1+(K-4)m_{2}-(K-3)m_{4}}
\end{equation*}
and
\begin{equation*}
\delta_{\mathbf{J}} =\frac{2m_{4} -
               4\delta_{\mathbf{S}}((K-3)m_{4}+2m_{2})-2c_{2}m_{2}(1-2m_{2}+m_{4})}%
               {2+ 4(K - 2)m_{2} + (K - 2)(K - 3)m_{4}}\,.
\end{equation*}
That the matrix in~\eqref{eq:inverseofinfo} indeed is the inverse of
the information matrix, can be verified by straightforward
multiplication of the matrices.

\section{Optimal Designs}
\label{sec:DoptDesign}
Without constraints on the design region the full factorial design,
which assigns equal weights $2^{-K}$ to each of the $2^K$ vertices, may
be used.  The information matrix of the full factorial design is equal
to the identity matrix $\mathbf{I}_p$, and the full factorial is
well-known to be optimal with respect to a variety of criteria
including $D$-optimality.  Hence, any design $\xi$ satisfying
$\mathbf{M}(\xi)=\mathbf{I}_p$ will be optimal on the unrestricted
design region and, by a majorization argument, also optimal on a 
restricted design region as long as the support of $\xi$ is
included in that design region.  Thus for finding an optimal design
$\xi^*$ on $\mathcal{X}$ it would be sufficient to show that
$\mathbf{M}(\xi^*)=\mathbf{I}_p$ or, equivalently, for a symmetric
invariant design $\bar\xi^*$ that $m_2(\bar\xi^*)=0$ and
$m_4(\bar\xi^*)=0$.

In particular, if $K\geq 6$ is even and $L=1$ we may take the regular
half fraction
$\bar\xi=2^{-(K-1)}\sum_{j=1}^{K/2}C(K,2j-1)\bar\xi_{2j-1}$ which is
uniform on all settings belonging to the odd orbits
$\mathcal{O}_1, \mathcal{O}_3, \ldots, \mathcal{O}_{K-1}$.  This half
fraction $\bar{\xi}$ has an information matrix equal to the identity
and is thus optimal on the restricted design region
$\mathcal{X}_{1,K-1}$.

In general, for larger $L$ the search for such designs may be more
complicated.  If $L$ becomes too large, then the constraints may
become so severe that the condition $\mathbf{M}(\xi)=\mathbf{I}_p$
cannot be met by any design on $\mathcal{X}_{L,K-L}$, and optimal
designs have to be characterized in another way.

Moreover, it would be desirable to reduce the number of support
points, i.e.\ the number of orbits for a symmetric invariant optimal
design.  Therefore, we first establish a result which provides optimal
designs supported on, at most, three symmetric orbits for small to
moderate thresholds $L$.
Then for $L$ up to a suitable threshold the following results shows
that there exist optimal designs on $\mathcal{X}_{L,K-L}$ which are
equally good as the full factorial design.  We start with the
situation where $L$ is equal to the threshold.  There symmetric
invariant designs turn out to be optimal which are supported on the
two symmetric orbits with minimal and maximal number of active rules
or with (nearly) half of the rules active, respectively.
\newpage
\begin{lemma}
  \label{lem:LequalB}
  Let either
  \begin{enumerate}
  \item\label{lem:LequalBeven} $K$ be even, $L=(K-\sqrt{3K-2})/{2}$,
    $\bar{w}_L=K/(2(3K-2))$, and
    $\bar{\xi}=\bar{w}_L\bar{\xi}_L+(1-2\bar{w}_L)\bar{\xi}_{K/2}+\bar{w}_L\bar{\xi}_{K-L}$,
    or
  \item\label{lem:LequalBodd}
    $K$ be odd, $L=(K-\sqrt{3K})/{2}$,
    $\bar{w}_L=(K-1)/(2(3K-1))$, and
    $\bar{\xi}=\bar{w}_L\bar{\xi}_L+
      (1/2-\bar{w}_L)\bar{\xi}_{(K-1)/2}
      +(1/2-\bar{w}_L)\bar{\xi}_{(K+1)/2}
      +\bar{w}_L\bar{\xi}_{K-L}$.
  \end{enumerate}
  Then the information matrix $\mathbf{M}(\bar{\xi})$ is equal to the
  identity $\mathbf{I}_p$.
\end{lemma}
The proof follows by straightforward calculation of
$m_2(\bar{\xi})=m_4(\bar{\xi})=0$, which establishes
$\mathbf{M}(\bar{\xi})=\mathbf{I}_p$.

The conditions of Lemma~\ref{lem:LequalB} are met only in rare cases.
For $K=3$ we recover the $2^3$ full factorial, and for $K=6$ the given
design is the $2^{6-1}$ fractional factorial on the odd orbits
mentioned above.  For $K=22$ we obtain an optimal design concentrated
on three orbits with $7$, $11$, and $15$ active rules, but with
unequal weights on the individual settings.  Similarly, for $K=27$ the
optimal design is supported on four orbits with $9$, $13$, $14$, and
$18$ active rules, where also the weights differ between the settings
of the outer and inner orbits.

For notational convenience we introduce the abbreviation $B_{K}$
for the threshold occurring in Lemma~\ref{lem:LequalB}:
\begin{equation*}
  B_{K} =
  \begin{cases}
    \frac{K-\sqrt{3K-2}}{2},&K\quad\text{even}\\
    \frac{K-\sqrt{3K}}{2},&K\quad\text{odd}\,.
  \end{cases}
\end{equation*}
When $L$ is less or equal to the threshold $B_K$, then, in general, at
least three symmetric orbits are required.
\begin{theorem}
  \label{thm:DoptEffAsFactorial}
  Let $L\leq B_K$, then there exist symmetric invariant designs
  $\bar{\xi}^*$ with $\mathbf{M}(\bar{\xi}^*)=\mathbf{I}_p$ which are
  supported on, at most, three symmetric orbits in
  $\mathcal{X}_{L,K-L}$.
\end{theorem}
In the corresponding proof in the appendix particular designs will be
constructed which include the outmost and the central symmetric orbits
$\tilde{\mathcal{O}}_L$ and $\tilde{\mathcal{O}}_{K/2}$ or
$\tilde{\mathcal{O}}_{(K-1)/2}$, respectively. For $K$ up to $22$ a
list of such designs is provided in Table~\ref{tab:optOuterOrbitSym}.
Note that there the index $c$ stands for a central orbit with $c=K/2$
or $c=(K-1)/2$, respectively.

All these designs are optimal because their information matrices
coincide with that of a full factorial design, which is known to be
optimal on the unrestricted design region $\{-1,+1\}^K$.  By a
majorization argument we may thus state the following result.
\begin{corollary}
  \label{cor:OuterDopt}
  The designs specified in Theorem~\ref{thm:DoptEffAsFactorial} are
  $D$-optimal.
\end{corollary}
However, the designs in Theorem~\ref{thm:DoptEffAsFactorial} and
Corollary~\ref{cor:OuterDopt} typically need not be unique.  For
example, if $L\leq B_K-1$, then the outmost orbit
$\tilde{\mathcal{O}}_L$ can be replaced by a less extreme one, i.e.\
$\tilde{\mathcal{O}}_k$ with $L<k\leq B_K$, more orbits can be
included by mixing different optimal designs, and even the central
orbit may be replaced as in half fractions on odd orbits when $K$ is a
multiple of four.

For $K=6$ we get $B_{K}=1$ for the threshold. Therefore the results
cannot be applied to the example with $L=2$. But as was mentioned
before the $2^{6-1}$ fractional factorial design on the odd orbits is
optimal for $L\leq 1$.

For narrower constraints ($L>B_K$) the full information can no longer
be retained, and the $D$-optimal designs will result in an information
matrix different from the identity, i.e.\ $m_2\neq 0$ and/or
$m_4\neq 0$.
\begin{theorem}
  \label{thm:innerOrbits}
  Let $B_{K} < L <K/2$.  Then the symmetric invariant design
  $\bar{\xi}^*=\bar{w}_L^*\bar{\xi}_L+(1-2\bar{w}_L^*)\bar{\xi}_{K/2}
  +\bar{w}_L^*\bar{\xi}_{K-L}$ in the case $K$ even and
  $\bar{\xi}^*=\bar{w}_L^*\bar{\xi}_L+(1/2-\bar{w}_L^*)\bar{\xi}_{(K-1)/2}
  +(1/2-\bar{w}_L^*)\bar{\xi}_{(K+1)/2}+\bar{w}_L^*\bar{\xi}_{K-L}$ in
  the case $K$ odd with optimized weight $\bar{w}_L^*$ is $D$-optimal.
\end{theorem}
From the proof given in the appendix it can be seen that the optimal design
specified in Theorem~\ref{thm:innerOrbits} is unique within the class
of symmetric invariant designs.

The weights $\bar{w}_{L}^*$ can be expressed as roots of polynomials
of degree three for $K$ even or five for $K$ odd, which arise from
setting the derivative of the criterion function
$\log\det(\mathbf{M}(\bar{\xi}))$ with respect to the weight equal to
zero. In general these roots are not rational numbers, and the optimal
designs from Theorem~\ref{thm:innerOrbits} cannot be directly realized
as exact designs.  However, they may well serve as benchmarks for
realistic candidates.

For the introductory example of $K=6$ rules, in which $L=2$ up to
$K-L=4$ rules can be active, the symmetric invariant optimal design is
supported by all possible orbits $\mathcal{O}_2$, $\mathcal{O}_3$ and
$\mathcal{O}_4$ of $\mathcal{X}_{2,4}$ with weights
$\bar{w}_2^*=\bar{w}_4^*=(45- 6\cdot\sqrt{37})/22\approx 0.3865$ and
$\bar{w}_3^*=1-2\bar{w}_2^*\approx 0.2270$.  The efficiency of this
design with respect to the full factorial design is
$\det(\mathbf{M}(\bar{\xi}^*))^{(1/p)}\approx 0.8854$.  The weights
for individual design points are $\bar{\xi}^*(\bm{x})\approx 0.0258$
for $\bm{x}\in\mathcal{O}_{2} \cup \mathcal{O}_{4}$ and
$\bar{\xi}^*(\bm{x})\approx 0.0113$ for $\bm{x}\in\mathcal{O}_{3}$.

Numerical values for optimal weights $\bar{w}_L^*$ are given in
Table~\ref{tab:optInnerOrbitSym}.  Note that also there the index $c$
stands for a central orbit.  The values in the tables were computed in
R~\citep{R:2018}.  For Table~\ref{thm:DoptEffAsFactorial} the weights
of the designs given in the proof of
Theorem~\ref{thm:DoptEffAsFactorial} were implemented. In the other
case, for Table~\ref{tab:optInnerOrbitSym}, the weights are roots of
polynomials, as was mentioned above. These were computed using the
polynom package~\citep*{VenablesHornikMaechler:2016}.  Optimality of
the resulting designs was checked using the equivalence theorem. In
all cases condition~\eqref{eq:equivalencetheorem} holds numerically
with a maximum error for $\psi(\bm{x})-p$ of order $10^{-12}$ or
smaller.

For the sake of clarity the results in the tables were rounded to four
digits.  This concerns especially the values in
Table~\ref{tab:optInnerOrbitSym}. Comparing the $D$-efficiency of the
rounded with the original values shows, that the loss is of order
$10^{-7}$ and hence negligible.

\section{Discussion}
\label{sec:discussion}
In the present paper we developed initial characterizations for
$D$-optimal designs in $K$-factorial models with binary predictors,
when the number of active factors is symmetrically bounded from below
and from above.

For mild to moderate constraints the obtained results have the same
information matrix as the full factorial design and have, hence
$100$\,\% efficiency.  This carries over also to other optimality
criteria based on the eigenvalues of the information matrix like $A$-
and $E$-optimality.  Conditions under which a fully efficient design
exists, i.e.\ a design with the information matrix $\mathbf{I}_{p}$,
are discussed by \cite{Harman:2008} in the context of Schur
optimality.

For strong constraints the restriction is so severe that the obtained
optimal designs do no longer have the same information matrix as the
full factorial. However, the resulting efficiencies, which are listed
in Table~\ref{tab:optInnerOrbitSym}, are still rather high. Also in
this situation the results can be extended to other optimality
criteria, but different weights have to be determined.

Further research is needed for dealing with asymmetric constraints, in
particular, in the case of narrow bounds or when central orbits are
excluded.  If the bounds are wide enough such that the lower bound is
below the threshold $B_K$ and the upper bound is above $K-B_K$, then
designs characterized in Theorem~\ref{thm:DoptEffAsFactorial} can be
used, where the bound is chosen as $\max\{L,K-U\}$ by majorization.

\section*{Acknowledgment}
This work was partly supported under DFG grant SCHW
531/15-4, while the first author was affiliated with the
University of Magdeburg.

\section*{Appendix A: Proofs}
For the proof of Lemma~\ref{lem:regularity}
we need some auxiliary results on the eigenvalues of
$\mathbf{M}_{22}-m_2^2\mathbf{J}_{C(K,2)}$
for a symmetric invariant design when $K\geq 4$.

\begin{lemma}
  \label{lem:eigenvalue3}
  Let 
  \begin{equation*}
  \lambda_{\bm{1}}= 
  1+2(K-2)m_{2} +\textstyle{\frac{1}{2}}(K-2)(K-3)m_{4} - \textstyle{\frac{1}{2}}K(K-1)m_{2}^2 
  \end{equation*}
  be the eigenvalue of $\mathbf{M}_{22}-m_2^2\mathbf{J}_{C(K,2)}$
  associated with the eigenvector $\bm{1}_{C(K,2)}$.
  Then $\lambda_{\bm{1}}>0$
  if and only if the support of $\bar{\xi}$ includes at least two
  distinct symmetric orbits.
\end{lemma}
\begin{proof}
  First note that for each invariant design $\bar{\xi}_k$ on a single
  orbit $\mathcal{O}_k$ we get by inserting the moments that the
  corresponding eigenvalue $\lambda_{\bm{1}}(\bar{\xi}_k)$ is zero.
  Let $\tilde{\xi}_{k}$ be the corresponding symmetric invariant
  design on the symmetric orbit $\tilde{\mathcal{O}}_{k}$. Then
  $m_{2}(\tilde{\xi}_k) = m_{2}(\bar{\xi}_{k})$,
  $m_{4}(\tilde{\xi}_k) = m_{4}(\bar{\xi}_{k})$ and, hence,
  $\lambda_{\bm{1}}(\tilde{\xi}_k)=\lambda_{\bm{1}}(\bar{\xi}_k)=0$.
  Consequently at least two distinct symmetric orbits are needed for
  $\lambda_{\bm{1}}>0$.

  Now, let
  $\bar{\xi}=\bar{w}_k\tilde{\xi}_k+\bar{w}_{\ell}\tilde{\xi}_{\ell}$
  be a symmetric invariant design on the symmetric orbits
  $\tilde{\mathcal{O}}_{k}$ and $\tilde{\mathcal{O}}_{\ell}$,
  $k<\ell\leq K/2$, $\bar{w}_k,\bar{w}_{\ell}>0$.  As
  $m_2(\bar{\xi}_k)$ is strictly decreasing in $k$ we have
  $m_2(\bar{\xi}_k)\neq m_2(\bar{\xi}_{\ell})$.  Thus
  $m_2(\bar{\xi})^2<\bar{w}_k m_2(\bar{\xi}_k)^2 + \bar{w}_{\ell}
  m_2(\bar{\xi}_{\ell})^2$ by the strict concavity of the quadratic
  function, which implies $\lambda_{\bm{1}}>0$.
\end{proof}

\begin{lemma}
  \label{lem:eigenvalue2}
  Let $\lambda_{\mathbf{S}}=1+(K-4)m_{2} - (K-3) m_{4}$
    be the eigenvalue of $\mathbf{M}_{22}-m_2^2\mathbf{J}_{C(K,2)}$
  associated with the remaining eigenvectors of 
  $\mathbf{S}_{K}\mathbf{S}_{K}^{\top}$
  orthogonal to $\bm{1}_{C(K,2)}$.
  Then $\lambda_{\mathbf{S}}>0$
  if and only if the support of $\bar{\xi}$ includes at least one
  symmetric orbit $\tilde{\mathcal{O}}_{k}$ 
  for which $0<k<K/2$.
\end{lemma}
\begin{proof}
  By inserting the moments we get for this eigenvalue
  \begin{equation*}
    \lambda_{\mathbf{S}}(\bar{\xi}_{k})
    = \frac{(2k-K)^2(K^2 - (2k-K)^2)}{K(K-1)(K-2)}\,,
  \end{equation*}
  which is equal to $0$ for $k=0$ or $k=K/2$.  Moreover,
  $\lambda_{\mathbf{S}}(\bar{\xi}_k)$ is a polynomial in $k$ of degree
  four, symmetric around $K/2$, and with negative leading term.  Hence,
  there cannot be any other root, and
  $\lambda_{\mathbf{S}}(\bar{\xi}_k)>0$ for all $0<k<K/2$.
\end{proof}

\begin{lemma}
  \label{lem:eigenvalue1}
  Let $\lambda_{\mathbf{I}} = 1-2m_{2}+m_{4}$ be the eigenvalue of
  $\mathbf{M}_{22}-m_2^2\mathbf{J}_{C(K,2)}$ associated with the
  remaining eigenvectors orthogonal to those of
  $\mathbf{S}_{K}\mathbf{S}_{K}^{\top}$.  Then
  $\lambda_{\mathbf{I}}>0$ if and only if the support of $\bar{\xi}$
  includes at least one symmetric orbit $\tilde{\mathcal{O}}_{k}$ for
  which $k>1$.
\end{lemma}
\begin{proof}
  By inserting the moments we see that
  $\lambda_{\mathbf{I}}(\bar{\xi}_k)$ is a polynomial in $k$ of degree
  four, symmetric around $K/2$, and with positive leading term, which
  is equal to zero for $k=0$ and $k=1$.  Hence, there cannot be any
  other root, and $\lambda_{\mathbf{I}}(\bar{\xi}_k)>0$ for all
  $1<k\leq K/2$.
\end{proof}

\begin{proof}[Proof of Lemma~\ref{lem:regularity}]
  We note that according to \cite{FreiseHollingSchwabe:2018prea} the
  matrix $\mathbf{M}_{11}$ associated with the main effects is
  nonsingular when at least two distinct symmetric orbits are involved
  in the design $\bar{\xi}$.  Hence, the requirement of
  $\mathbf{M}_{11}$ to be nonsingular does not impose any additional
  condition besides those of Lemmas~\ref{lem:eigenvalue3}
  to~\ref{lem:eigenvalue1}.  Hence, the assertion follows from these
  lemmas.
\end{proof}

\begin{proof}[Proof of Theorem~\ref{thm:DoptEffAsFactorial}]
  For $K=2$ and $K=3$ we have $B_K<1$, such that the full factorial
  design can serve as the asserted symmetric invariant design.
 
  Let $K\geq 4$.  If $L=B_{K}$, then the design of
  Lemma~\ref{lem:LequalB} can be used.  Let $L<B_{K}$ and let $K$ be
  even. Then choose $\ell$ such that
  \begin{equation*}
   B_{K} = \frac{K-\sqrt{3K-2}}{2} \leq \ell \leq
  \frac{K-\sqrt{K}}{2}\,. 
  \end{equation*}
  Such $\ell$ always exists (for $ K\leq 8$ see
  Table~\ref{tab:optOuterOrbitSym};  note that
  $\sqrt{3K-2}-\sqrt{K}\geq 2$ for $K\geq 10$).

  Let designs $\bar{\xi}_{(L)}$ and $\bar{\xi}_{(\ell)}$ be defined as
  \begin{equation}
    \label{eq:2}
    \bar{\xi}_{(L)}=\bar{w}_L\bar{\xi}_L+(1-2\bar{w}_L)\bar{\xi}_{K/2}
    + \bar{w}_L\bar{\xi}_{K-L}
    \quad\text{with}\quad
    \bar{w}_{L} = \frac{K}{2(2L-K)^2}
  \end{equation}
  and
  \begin{equation}
    \label{eq:3}
    \bar{\xi}_{(\ell)}=\bar{w}_{\ell}\bar{\xi}_{\ell}+(1-2\bar{w}_{\ell})\bar{\xi}_{K/2}
    + \bar{w}_{\ell}\bar{\xi}_{K-\ell}
    \quad\text{with}\quad
    \bar{w}_{\ell} = \frac{K}{2(2\ell-K)^2}\,.
  \end{equation}
  According to \cite{FreiseHollingSchwabe:2018prea} the moments
  $m_{2}(\bar{\xi}_{(L)})$ and $m_{2}(\bar{\xi}_{(\ell)})$ are equal
  to $0$, because $L,\ell<(K-\sqrt{K})/2$.  Next we consider the
  convex combination
  \begin{equation*}
    \bar{\xi}^{*} = \alpha \bar{\xi}_{(L)} + (1-\alpha)\bar{\xi}_{(\ell)}
    \quad\text{with}\quad
    \alpha= \frac{3K-2-(2\ell-K)^2}{4(\ell-L)(K - L - \ell)}\,.
  \end{equation*}
  Also for the symmetric invariant design $\bar{\xi}^*$ we have
  $m_{2}(\bar{\xi}^*) = 0$. Further we obtain
  \begin{equation*}
    m_{4}(\bar{\xi}^*)
    =  \frac{4 \alpha(\ell-L)(K - L - \ell)+ (2\ell-K)^2 -
      (3K-2)}{(K-1)(K-2)(K-3)}
    = 0\,,
  \end{equation*}
  which establishes the result for even $K$.
  
  For $K$ odd the proof is similar.  For $L<B_{K}$ choose $\ell$ as
  above with the corresponding value for $B_{K}$ and the designs
  \begin{equation*}
    \bar{\xi}_{(L)}=\bar{w}_L\bar{\xi}_L+(\textstyle{\frac{1}{2}}-\bar{w}_L)\bar{\xi}_{(K-1)/2}
    +(\textstyle{\frac{1}{2}}-\bar{w}_L)\bar{\xi}_{(K+1)/2}+ \bar{w}_L\bar{\xi}_{K-L}
  \end{equation*}
  and
  \begin{equation*}
    \bar{\xi}_{(\ell)}=\bar{w}_{\ell}\bar{\xi}_{\ell}+(\textstyle{\frac{1}{2}}-\bar{w}_{\ell})\bar{\xi}_{(K-1)/2}
    +(\textstyle{\frac{1}{2}}-\bar{w}_{\ell})\bar{\xi}_{(K+1)/2}+ \bar{w}_{\ell}\bar{\xi}_{K-\ell}
  \end{equation*}
  with corresponding weights
  \begin{equation*}
    \bar{w}_{L}
    =\frac{K-1}{2((2L-K)^2-1)}
    \quad\text{and}\quad
    \bar{w}_{\ell}
    =\frac{K-1}{2((2\ell-K)^2-1)}\,.
  \end{equation*}
  Again such $\ell$ always exists (for $K\leq 7$ see
  Table~\ref{tab:optOuterOrbitSym}; for $K \geq 9$ note that
  $\sqrt{3K}-\sqrt{K}\geq 2$) and
  $m_{2}(\bar{\xi}_{(L)})=m_{2}(\bar{\xi}_{(\ell)})=0$.

  Then the convex combination
  \begin{equation*}
    \bar{\xi}^{*} = \alpha \bar{\xi}_{(L)} + (1-\alpha)\bar{\xi}_{(\ell)}
    \quad\text{with}\quad
    \alpha
    = \frac{3K-(2\ell-K)^2}{4(\ell-L)(K-L-\ell)}
  \end{equation*}
  yields $m_{4}(\bar{\xi}^*)=0$, and the result follows.
\end{proof}

For the proof of Theorem~\ref{thm:innerOrbits} we will make use of the
celebrated Kiefer-Wolfowitz equivalence theorem
\citep{KieferWolfowitz:1960}.  Therefore we first investigate the
sensitivity function (functional derivative).

\begin{lemma}
  \label{lem:sensitivitypoly}
  Let $\bar{\xi}$ be a symmetric invariant design. 
  Then the sensitivity function 
  $\psi(\bm{x})=\bm{f}(\bm{x})^{\top}\mathbf{M}(\bar{\xi})^{-1}\bm{f}(\bm{x})$ 
  is constant on the orbits,
  $\psi(\bm{x})=\tilde{\psi}(k)$ for $\bm{x}\in\mathcal{O}_k$, say,
  and the function $\tilde{\psi}$ is a polynomial of degree at most $4$, 
  which is symmetric with respect to $K/2$ $(\tilde{\psi}(K-k)=\tilde{\psi}(k))$.
\end{lemma}
\begin{proof}
  Using the inverse of the information matrix
  in equation~\eqref{eq:inverseofinfo} we obtain for the sensitivity function 
  \begin{equation}
  \label{eq:sensitivity}
  \psi(\bm{x})
  = c_{0} -2c_{2}\bm{\tilde{x}}^\top\bm{1}_{C(K,2)}
  +\bm{x}^\top\mathbf{M}_{11}^{-1}\bm{x}
  +\bm{\tilde{x}}^\top\mathbf{C}_{22}\bm{\tilde{x}}\,.
\end{equation}
Note that $\bm{x}^\top\bm{x} = K$ and $\bm{\tilde{x}}^\top\bm{\tilde{x}} =
K(K-1)/2$. Further, for $\bm{x} \in \mathcal{O}_{k}$, we get
\begin{equation*}
  \bm{x}^\top\bm{1}_{K} = 2 k - K\,,
  \quad
  \bm{\tilde{x}}^\top\bm{1}_{C(K,2)} = ((2k-K)^2 - K)/2
\end{equation*}
and
\begin{equation*}
  \bm{\tilde{x}}^\top\mathbf{S}_{K}\mathbf{S}_{K}^\top\bm{\tilde{x}}
  = (K - 2)(2 k-K)^{2} + K\,.
\end{equation*}
This yields
\begin{equation}
  \label{eq:sensitivityink}
  \psi(\bm{x})
  = a_{4}(2k-K)^4+ a_{2}(2k-K)^{2}+a_{0}
\end{equation}
with coefficients
\begin{align*}
  a_{0}&=c_{0} + \frac{K}{1-m_{2}} +\frac{K(K-1)}{2(1-2m_{2}+m_{4})}
         -\frac{\delta_{\mathbf{S}}K}{1-2m_{2}+m_{4}}
         -\frac{\delta_{\mathbf{J}}K^2}{4(1-2m_{2}+m_{4})}\\
  a_{2}&=-\left(c_{2}
  + \frac{m_{2}}{(1-m_{2})(1+(K-1)m_{2})}
         +\frac{\delta_{\mathbf{S}}(K-2)}{1-2m_{2}+m_{4}}
         -\frac{\delta_{\mathbf{J}}K}{2(1-2m_{2}+m_{4})}\right)\\
  a_{4}&=-\frac{\delta_{\mathbf{J}}}{4(1-2m_{2}+m_{4})}\,.
\end{align*}
Hence, $\tilde\psi$ is a polynomial of degree four in $k$.
The symmetry around $K/2$ follows, since only even powers of
$k$ occur.
\end{proof}
\begin{lemma}
  \label{lem:leadingterm}
  Let $B_{K}<L<K/2$ and $\bar{\xi}^*$ an optimal symmetric invariant
  design on $\mathcal{X}_{L,K-L}$. Then the orbitwise sensitivity
  function $\tilde{\psi}$ has a positive leading term for the fourth
  order monomial $k^4$.
\end{lemma}

\begin{proof}
  We will use the same notation as in Lemma~\ref{lem:sensitivitypoly}
  and its proof.

  Let the coefficient of the fourth order monomial in $\tilde{\psi}$
  be nonpositive or equivalently let $a_{4}\leq 0$. Then the
  function $\tilde{\psi}$ as a function in $k$ has either a single maximum at
  $k=K/2$, two (symmetric) maxima outside $(L,K-L)$ (respectively at
  $k=L$ and $k=K-L$ for the admitted orbits), two (symmetric) maxima
  inside $(L,K-L)$, or is constant.

  In the first two cases $\bar{\xi}^*$ is supported on one symmetric
  orbit only.  It follows from Lemma~\ref{lem:regularity}, that the
  information matrix has to be singular. But in this case the function
  $\tilde{\psi}$ would not be defined, which is a contradiction.

  In the last case, if the orbitwise sensitivity is constant, we have
  $\tilde{\psi}(k)\leq p$ for all $k\in\{0,\ldots,K\}$ and,
  consequently, $\bar{\xi}^*$ is optimal on the unrestricted design
  region $\mathcal{X}_{0,K}$.
  
  In the third case, i.e.\ two maxima inside of $(L,K-L)$, 
  the optimal invariant design $\bar{\xi}^*$ has all its weight on
  either one or two symmetric orbits.  If these orbits do not satisfy
  the conditions in Lemma~\ref{lem:regularity}, the information
  matrix would be singular, which leads to a contradiction.  Otherwise
  the design is optimal on $\mathcal{X}_{0,K}$, with the same argument
  as for the constant case.

  Now let $\bar{\xi}^*$ be optimal on $\mathcal{X}_{0,K}$. Then its
  information matrix is $\mathbf{I}_{p}$. It follows that
  $m_{2}=m_{4} = 0$ and thus
  \begin{equation}
    \label{eq:1}
    \sum_{k=L}^{K-L}\bar{w}_{k}(2 k - K)^2  = K
    \quad\text{and}\quad
    \sum_{k=L}^{K-L}\bar{w}_{k}(2k-K)^4 - K^2 = 2K(K-1)\,.
  \end{equation}
  The left-hand sides of these two equations can be interpreted as
  expectation and variance, respectively, of a discrete random
  variable taking values $(2k-K)^2$, $k=L,\ldots,K-L$.  An upper bound
  for the variance is given in \citet{Muilwijk:1966a} \citep[see
  also][]{BhatiaDavis:2000}. This yields for the variance
  \begin{equation*}
    \sum_{k=L}^{K-L}\bar{w}_{k}(2k-K)^4 - K^2
    \leq ((2L-K)^2 - K)(K - R_{K})\,,
  \end{equation*}
  where $R_{K}=0$ for $K$ even and $R_{K} = 1$ for $K$ odd.  Since
  $B_{K}<L$ it follows that
  \begin{equation*}
    ((2L-K)^2 - K)(K - R_{K})
    < ((2B_{K}-K)^2 - K)(K - R_{K})
    = 2K(K-1)\,,
  \end{equation*}
  which is in contradiction to~\eqref{eq:1}.  Hence $a_{4}$ and
  consequently the leading coefficient has to be positive.
\end{proof}
\begin{proof}[Proof of Theorem~\ref{thm:innerOrbits}]
  Under the assumptions of Theorem~\ref{thm:innerOrbits}, the
  orbitwise sensitivity function $\tilde{\psi}$ of the optimal design
  $\bar{\xi}^*$ is a polynomial of degree four with positive leading
  term by Lemma~\ref{lem:sensitivitypoly} and \ref{lem:leadingterm}.
  Then, in view of the fundamental theorem of algebra, the equality
  $\tilde\psi(k)=p$ can only have at most four distinct roots.
  Because of the symmetry of the sensitivity function with respect to
  $K/2$ (cf.\ Lemma~\ref{lem:sensitivitypoly}) the optimal design has
  thus to be concentrated on at most two symmetric orbits.  In order
  to fulfill the condition $\tilde\psi(k)\leq p$ for all
  $k=L,\ldots,K-L$, imposed by the equivalence theorem on the optimal
  design $\bar{\xi}^*$, these symmetric orbits can only be the outmost
  orbit $\mathcal{O}_{L}\cup\mathcal{O}_{K-L}$ on the boundaries and
  the central orbit $\mathcal{O}_{K/2}$ for $K$ even and
  $\mathcal{O}_{(K-1)/2}\cup\mathcal{O}_{(K+1)/2}$ for $K$ odd,
  respectively.

  On the other hand the nonsingularity condition of
  Lemma~\ref{lem:regularity} requires that the optimal design
  $\bar{\xi}^*$ has to be supported by at least two symmetric orbits
  and, hence, $\bar{\xi}^*$ is of the form specified in the Theorem.
  Finally only the weights have to be optimized given the two
  symmetric orbits.
 \end{proof}

 \newpage
 \section*{Appendix B: Tables}
\begin{table}[ht]
\centering
\caption{Symmetric invariant $D$-optimal designs for wide bounds 
from Theorem \ref{thm:DoptEffAsFactorial}} 
\label{tab:optOuterOrbitSym}
\begin{tabular}{rrrrrrrr}
  \hline
  $K$ & $L$  & $\ell$ & $c$ & $\bar{w}_{L}^*$
  & $\bar{w}_{\ell}^*$
  & $\bar{w}_{c}^*$ & $B_{K}$ \\ 
  \hline
 4 & 0 & 1 & 2 &0.0625 & 0.2500 & 0.3750 & 0.42 \\ 
 \hline
 5 & 0 & 1 & 2 &0.0312 & 0.1562 & 0.3125 & 0.56 \\ 
 \hline
 6 & 0 & 1 & 3&$-$ & 0.1875 & 0.6250 & 1\phantom{.00} \\ 
   & 1 & $-$ & 3&0.1875 & $-$ & 0.6250 & 1\phantom{.00} \\ 
 \hline
 7 & 0 & 2 & 3&0.0187 & 0.2625 & 0.2188 & 1.21 \\ 
   & 1 & 2 & 3&0.0938 & 0.0938 & 0.3125 & 1.21 \\ 
 \hline
 8 & 0 & 2 & 4&0.0078 & 0.2188 & 0.5469 & 1.65 \\ 
   & 1 & 2 & 4&0.0333 & 0.1750 & 0.5833 & 1.65 \\ 
 \hline
 9 & 0 & 2 & 4&0.0018 & 0.1607 & 0.3375 & 1.90 \\ 
   & 0 & 3 & 4&0.0125 & 0.3750 & 0.1125 & 1.90 \\ 
   & 1 & 2 & 4&0.0069 & 0.1528 & 0.3403 & 1.90 \\ 
   & 1 & 3 & 4&0.0375 & 0.2750 & 0.1875 & 1.90 \\ 
  \hline
10 & 0 & 3 & 5&0.0071 & 0.2679 & 0.4500 & 2.35 \\ 
   & 1 & 3 & 5&0.0195 & 0.2344 & 0.4922 & 2.35 \\ 
   & 2 & 3 & 5&0.0833 & 0.1250 & 0.5833 & 2.35 \\ 
   \hline
11 & 0 & 3 & 5&0.0035 & 0.1910 & 0.3056 & 2.63 \\ 
   & 1 & 3 & 5&0.0089 & 0.1786 & 0.3125 & 2.63 \\ 
   & 2 & 3 & 5&0.0347 & 0.1389 & 0.3264 & 2.63 \\ 
   \hline
12 & 0 & 4 & 6&0.0059 & 0.3223 & 0.3438 & 3.08 \\ 
   & 1 & 4 & 6&0.0129 & 0.2946 & 0.3850 & 3.08 \\ 
   & 2 & 4 & 6&0.0352 & 0.2344 & 0.4609 & 3.08 \\ 
   & 3 & 4 & 6&0.1500 & 0.0375 & 0.6250 & 3.08 \\ 
   \hline
22 & 0 & 7 & 11&$-$ & 0.1719 & 0.6562 & 7\phantom{.00} \\ 
   & 0 & 8 & 11&0.0014 & 0.2865 & 0.4242 & 7\phantom{.00} \\ 
   & 1 & 7 & 11&$-$ & 0.1719 & 0.6562 & 7\phantom{.00} \\ 
   & 1 & 8 & 11&0.0021 & 0.2821 & 0.4317 & 7\phantom{.00} \\ 
   & 2 & 7 & 11&$-$ & 0.1719 & 0.6562 & 7\phantom{.00} \\ 
   & 2 & 8 & 11&0.0033 & 0.2758 & 0.4417 & 7\phantom{.00} \\ 
   & 3 & 7 & 11&$-$ & 0.1719 & 0.6562 & 7\phantom{.00} \\ 
   & 3 & 8 & 11&0.0055 & 0.2667 & 0.4557 & 7\phantom{.00} \\ 
   & 4 & 7 & 11&$-$ & 0.1719 & 0.6562 & 7\phantom{.00} \\ 
   & 4 & 8 & 11&0.0098 & 0.2521 & 0.4762 & 7\phantom{.00} \\ 
   & 5 & 7 & 11&$-$ & 0.1719 & 0.6562 & 7\phantom{.00} \\ 
   & 5 & 8 & 11&0.0198 & 0.2263 & 0.5077 & 7\phantom{.00} \\ 
   & 6 & 7 & 11&$-$ & 0.1719 & 0.6562 & 7\phantom{.00} \\ 
   & 6 & 8 & 11&0.0481 & 0.1719 & 0.5600 & 7\phantom{.00} \\ 
   & 7 & $-$ & 11&0.1719 & $-$ & 0.6562 & 7\phantom{.00} \\ 
   & 7 & 8 & 11&0.1719 & $-$ & 0.6562 & 7\phantom{.00} \\ 
   \hline
\end{tabular}
\end{table}
\clearpage
\begin{table}[ht]
  \centering
  \caption{Symmetric invariant $D$-optimal designs for narrow bounds
    from Theorem~\ref{thm:innerOrbits}}
  \label{tab:optInnerOrbitSym}
  \begin{tabular}{rrrrrrr}
    \hline
    $K$ & $L$ & $c$ & $\bar{w}_{L}^*$ & $\bar{w}_{c}^*$ & $D$-Efficiency & $B_{K}$ \\ 
    \hline
    4  & 1  & 2  & 0.2993 & 0.4015 & 0.8892 & 0.42 \\   \hline
    5  & 1  & 2  & 0.1939 & 0.3061 & 0.9725 & 0.56 \\    \hline
    6  & 2  & 3  & 0.3865 & 0.2270 & 0.8854 & 1\phantom{.00} \\   \hline
    7  & 2  & 3  & 0.2798 & 0.2202 & 0.9682 & 1.21 \\     \hline
    8  & 2  & 4  & 0.2282 & 0.5435 & 0.9960 & 1.65 \\ 
       & 3  & 4  & 0.4212 & 0.1576 & 0.8846 & 1.65 \\     \hline
    9  & 3  & 4  & 0.3461 & 0.1539 & 0.9660 & 1.90 \\     \hline
    10 & 3  & 5  & 0.2744 & 0.4512 & 0.9926 & 2.35 \\ 
       & 4  & 5  & 0.4397 & 0.1205 & 0.8863 & 2.35 \\     \hline
    11 & 3  & 5  & 0.1969 & 0.3031 & 0.9985 & 2.63 \\ 
       & 4  & 5  & 0.3903 & 0.1097 & 0.9640 & 2.63 \\     \hline
    12 & 4  & 6  & 0.3188 & 0.3624 & 0.9905 & 3.08 \\ 
       & 5  & 6  & 0.4513 & 0.0975 & 0.8892 & 3.08 \\     \hline
    13 & 4  & 6  & 0.2313 & 0.2687 & 0.9973 & 3.38 \\ 
       & 5  & 6  & 0.4178 & 0.0822 & 0.9622 & 3.38 \\     \hline
    14 & 4  & 7  & 0.1911 & 0.6177 & 0.9999 & 3.84 \\ 
       & 5  & 7  & 0.3582 & 0.2836 & 0.9891 & 3.84 \\ 
       & 6  & 7  & 0.4591 & 0.0818 & 0.8924 & 3.84 \\     \hline
    15 & 5  & 7  & 0.2655 & 0.2345 & 0.9965 & 4.15 \\ 
       & 6  & 7  & 0.4352 & 0.0648 & 0.9607 & 4.15 \\     \hline
    16 & 5  & 8  & 0.2146 & 0.5707 & 0.9994 & 4.61 \\ 
       & 6  & 8  & 0.3904 & 0.2191 & 0.9879 & 4.61 \\ 
       & 7  & 8  & 0.4648 & 0.0704 & 0.8957 & 4.61 \\     \hline
    17 & 6  & 8  & 0.2987 & 0.2013 & 0.9959 & 4.93 \\ 
       & 7  & 8  & 0.4469 & 0.0531 & 0.9595 & 4.93 \\     \hline
    18 & 6  & 9  & 0.2385 & 0.5229 & 0.9990 & 5.39 \\ 
       & 7  & 9  & 0.4149 & 0.1702 & 0.9868 & 5.39 \\ 
       & 8  & 9  & 0.4691 & 0.0618 & 0.8990 & 5.39 \\   \hline
    19 & 6  & 9  & 0.1842 & 0.3158 & 0.9999 & 5.73 \\ 
       & 7  & 9  & 0.3301 & 0.1699 & 0.9954 & 5.73 \\ 
       & 8  & 9  & 0.4551 & 0.0449 & 0.9586 & 5.73 \\     \hline
    20 & 7  & 10 & 0.2624 & 0.4751 & 0.9987 & 6.19 \\ 
       & 8  & 10 & 0.4325 & 0.1349 & 0.9858 & 6.19 \\ 
       & 9  & 10 & 0.4725 & 0.0550 & 0.9021 & 6.19 \\     \hline
    21 & 7  & 10 & 0.2028 & 0.2972 & 0.9997 & 6.53 \\ 
       & 8  & 10 & 0.3590 & 0.1410 & 0.9950 & 6.53 \\ 
       & 9  & 10 & 0.4611 & 0.0389 & 0.9580 & 6.53 \\     \hline
    22 & 8  & 11 & 0.2861 & 0.4278 & 0.9984 & 7\phantom{.00} \\ 
       & 9  & 11 & 0.4451 & 0.1098 & 0.9848 & 7\phantom{.00} \\ 
       & 10 & 11 & 0.4752 & 0.0496 & 0.9051 & 7\phantom{.00} \\ 
    \hline
  \end{tabular}
\end{table}

\end{document}